\documentclass[reqno,11pt,centertags]{amsart}
\usepackage{amsmath,amsthm,amscd,amssymb,latexsym,upref}
\date{\today}

\input epsf
\usepackage{epsfig}
\usepackage[T2A,OT1]{fontenc}
\usepackage[ot2enc]{inputenc}
\usepackage[russian,english]{babel}

\usepackage{epic,eepicemu}

\newcommand{\bbD}{{\mathbb{D}}}

\newcommand{\bbZ}{{\mathbb{Z}}}

\newcommand{\cH}{{\mathcal{H}}}
\newcommand{\bbT}{{\mathbb{T}}}

\newcommand{\cL}{{\mathcal{L}}}
\newcommand{\cP}{{\mathcal{P}}}
\newcommand{\cQ}{{\mathcal{Q}}}

\newcommand{\cE}{{\mathcal{E}}}

\newcommand{\cM}{{\mathcal{M}}}
\newcommand{\cN}{{\mathcal{N}}}

\newcommand{\fe}{{\mathfrak{e}}}
\newcommand{\ff}{{\mathfrak{f}}}

\newcommand{\fA}{{\mathfrak{A}}}

\renewcommand{\a}{\alpha}
\newcommand{\aone}{\a_{-1}}
\newcommand{\bara}{\overline\alpha_{-1}}
\renewcommand{\Re}{\text{\rm Re}}

\newcommand{\tr}{\text{\rm tr}}

\newcommand{\Sz}{{\mathbf{Sz}}}

\newcommand{\GI}{{\mathbf{GI}}}

\newcommand{\bg}{\begin}
\newcommand{\eq}{equation}
\newcommand{\ovl}{\overline}
\allowdisplaybreaks \numberwithin{equation}{section}
\newtheorem{theorem}{Theorem}[section]

\newtheorem{lemma}[theorem]{Lemma}
\newtheorem{proposition}[theorem]{Proposition}
\newtheorem{corollary}[theorem]{Corollary}

\theoremstyle{definition}
\newtheorem{definition}[theorem]{Definition}

\newtheorem{remark}[theorem]{Remark}
\newtheorem{problem}[theorem]{Problem}

\newtheorem{example}[theorem]{Example}

\begin{document}

\title
[Scattering for CMV matrices] {Scattering theory for CMV matrices:
uniqueness, Helson--Szeg\H{o} and Strong Szeg\H{O} theorems}
\author[L. Golinskii, A. Kheifets, F. Peherstorfer,
 and P. Yuditskii]{L. Golinskii, A. Kheifets$^*$, \fbox{F. Peherstorfer}$^{**}$,
 and P. Yuditskii$^{**}$}
\thanks{$^*$ The work was partially supported by
the University of Massachusetts Lowell Research and Scholarship Grant,
project number: H50090000000010}

\thanks{$^{**}$The work was partially supported by the Austrian Science
Found FWF, project number: P20413--N18}

\date{\today}

\begin{abstract}
We develop a scattering theory for CMV matrices, similar to the
Faddeev--Marchenko theory. A necessary and sufficient condition is
obtained for the uniqueness of the solution of the inverse
scattering problem. We also obtain two sufficient conditions for the
uniqueness, which are connected with the Helson--Szeg\H o and the
Strong Szeg\H o theorems. The first condition is given in terms of
the boundedness of a transformation operator associated to the CMV
matrix. In the second case this operator has a determinant. In both
cases we characterize Verblunsky parameters of the CMV matrices,
corresponding spectral measures and scattering functions.
\end{abstract}

\maketitle

\section{Introduction}

To a given collection of numbers $\{\a_n\}_{n\ge 0}$ in  the open
unit disk $\bbD$, called the {\it Verblunsky coefficients}, and
$\aone$ in the unit circle $\bbT$, we define the CMV matrix
$\fA=\fA_{od}\fA_{e}$, where
\begin{equation*}\label{matrixform0}
\fA_{od}= \begin{bmatrix} -\aone & & & \\ &
A_1& & \\
   & & A_3 & \\
  &  & & \ddots
    \end{bmatrix}
  , \qquad
\fA_{e}=\begin{bmatrix}
    A_0 & & \\ &
     A_2 & \\
    & & \ddots
    \end{bmatrix}
    ,
\end{equation*}
and the $A_k$'s are the $2\times 2$ unitary matrices
$$
A_k=\begin{bmatrix}
 \overline\a_k&\rho_k\\ \rho_k&-\a_k
    \end{bmatrix},\quad \rho_k=\sqrt{1-|\a_k|^2}.
     $$
Unlike the standard convention \cite[p. 265]{Sopuc}, we do not fix
the value $\a_{-1}~=~-1$. Our reasons will become clear later on.

Note that $\fA$ is a unitary operator on $l^2(\bbZ_+)$. The initial
vector $e_0$ of the standard basis is cyclic for $\fA$. Indeed, by
the definition for $n=0,1,\ldots$
\begin{equation}\label{gi4}
\begin{split}
    \fA\{e_{2n}\rho_{2n}-e_{2n+1} \ovl\alpha_{2n}\}=&\ e_{2n+1}
\ovl\alpha_{2n+1}+e_{2n+2}\rho_{2n+1}\\
\fA^{-1}\{e_{2n+1}\rho_{2n+1}-e_{2n+2}\alpha_{2n+1}\}=&\ e_{2n+2}
\alpha_{2n+2}+e_{2n+3}\rho_{2n+2}\\
\fA^{-1} e_0 =&-\ovl\alpha_{-1}( e_0 \alpha_0 + e_1\rho_0).
\end{split}
\end{equation}
That is, acting in turn by $\fA^{-1}$ and $\fA$ on $e_0$ and
taking the linear combinations, we can get
any vector of the standard basis. CMV
matrices were introduced in \cite{CMV2}. More recent surveys on this
topic are \cite{Sopuc, sim2, kine}.

\subsection{Spectral Characteristics}

Since $\fA$ is a unitary operator, then the following function
\begin{equation}\label{gi5}
   R(z):=\left\langle\frac{\fA+z}{\fA-z}e_0,e_0\right\rangle=
   \int\limits_{\bbT}\frac{t+z}{t-z}\,\sigma(dt)
\end{equation}
has a nonnegative real part in the unit disk, which yields  the
integral formula in (\ref{gi5}). Measure $\sigma=\sigma(\fA)$ is
called the {\it spectral measure} of $\fA$ with respect to the
cyclic vector $e_0$. The standard Lebesgue decomposition is
 \bg{\eq}\label{lebdec}
 \sigma(dt)=w(t)m(dt)+\sigma_s(dt)
  \end{\eq}
where $m(dt)$ is the normalized Lebesgue measure, and $\sigma_{s}$
is the singular component. We will say that $\fA$ is absolutely
continuous if $\sigma_s=0$. Note that
$$ R(0)= \langle
e_0,e_0\rangle=\int\limits_{\bbT}\sigma(dt)=1,$$ so $\sigma$ is a
{\it probability measure}.

We define function $\phi$ by the equation
\begin{equation}\label{0001}
\phi(z)= \a_{-1}\frac{1-R(z)}{1+R(z)}\,, \qquad
R(z)=\frac{1-\bara\phi(z)}{1+\bara\phi(z)}\,.
\end{equation}
Then $|\phi|\le 1$, $\phi(0)=0$. \if{Moreover, for any function
$\cE$ from the Schur class (the unit ball of $H^\infty$)

with some probability measure $\sigma_\cE$.}\fi An important
relation is
 \bg{\eq}\label{wrealr}
w(t)=\Re R(t)=\frac{1-|\phi(t)|^2}{|1+\bara\phi(t)|^2}
 \end{\eq}
a.e. on $\bbT$.

The spectral measure $\sigma$ is uniquely determined from the CMV
matrix $\fA$ by \eqref{gi5}. Conversely, by the first formula in
\eqref{0001}, the measure $\sigma$ uniquely defines
$\ovl\a_{-1}\phi$. Hence, to recover $\phi$ (and by that $\a_n$), we
need to know $\a_{-1}$. Therefore, the pair $\{\sigma, \a_{-1}\}$,
not just $\sigma$, determines uniquely the CMV matrix $\fA$. That is
why we consider the pair $\{\sigma, \a_{-1}\}$ as the spectral data.

The one-to-one correspondences
$$
\fA\longleftrightarrow \{ \sigma, \a_{-1}\} \longleftrightarrow \{
R, \a_{-1}\} \longleftrightarrow \{\phi,  \a_{-1}\}
$$
are studied in the theory of orthogonal polynomials on the unit
circle (OPUC) \cite{Sopuc} and in the Schur analysis \cite{schur}.

\subsection{Direct scattering}

By definition, the matrix $\fA$ is in the {\it Szeg\H{o} class},
$\fA\in \Sz$, if $\sum |\a_k|^2<\infty$. It is known that $\fA\in
\Sz$ if and only if the spectral measure $\sigma$ is of the form
 \begin{equation}\label{resm}
    \fA\in\Sz\Leftrightarrow \log w\in L^1,
 \end{equation}
see \cite[Theorem 2.3.1]{Sopuc}. The standard fact from the theory
of Hardy classes reads that assumption (\ref{resm}) yields
\bg{\eq}\label{0002} w(t)=|D(t)|^2
\end{\eq}
a.e., where $D$ is a boundary value of an outer $H^2$ function,
$D(0)>0$. $D$ is known as the {\it Szeg\H{o} function}.
By the Szeg\H{o} theorem
\begin{equation}\label{szorig}
D(0)=\prod_{k=0}^\infty \rho_k.
\end{equation}
It follows from \eqref{wrealr} that
$$ \fA\in\Sz\Leftrightarrow \log (1-|\phi|^2)\in L^1,$$
so an outer function $\psi$, which satisfies
 \bg{\eq}\label{psifun}
 |\psi(t)|^2+|\phi(t)|^2=1, \qquad \psi(0)>0,
 \end{\eq}
is well defined, uniquely determined by $\phi$. By \eqref{wrealr}
 \bg{\eq}\label{wpsi}
  w(t)=\left|\frac{\psi(t)}{1+\bara\phi(t)}\right|^2
\end{\eq}
a.e. Hence $D$ is of the form
 \bg{\eq}\label{0003}
D(z)=\frac{\psi(z)}{1+\bara\phi(z)}, \qquad
\psi(0)=D(0)=\prod_{k=0}^\infty \rho_k.
\end{\eq}
\begin{definition}\label{scatfun}
The {\it scattering function} of $\fA$ is defined as
\begin{equation}\label{desssim}
    s(t)=-\bara\frac{D(t)}{\overline{D(t)}}
    =-\frac{\psi(t)}{\overline\psi(t)}\frac{\bara+\overline\phi(t)}{1+\bara\phi(t)}
    ,\quad t\in\mathbb T.
\end{equation}
Note that $|s(t)|=1$ a.e. on $\bbT$.
\end{definition}


In the Faddeev--Marchenko theory the scattering function appears as
a coefficient in the leading term of certain asymptotics. In our
context we have
\begin{theorem}\label{th1}
Let $\fA\in \Sz$. Then there exists a
unique generalized eigenvector $\Psi(t)=\{\Psi_n(t)\}_{n=0}^\infty$
such that
\begin{equation}\label{gi2}
    \begin{bmatrix}\Psi_0(t) & \Psi_1(t) & \ldots\end{bmatrix}
    \fA=
   t \begin{bmatrix}\Psi_0(t) & \Psi_1(t) & \ldots\end{bmatrix},\ t\in \bbT,
\end{equation}
and the following asymptotics holds in $L^2$--norm
\begin{equation}\label{gi3}
    \Psi_{2n}(t)=t^n+o(1),\quad \Psi_{2n+1}(t)=\overline{s(t)} t^{-n-1}+o(1), \ n\to
    \infty.
\end{equation}
\end{theorem}
Theorem \ref{th1} is a restatement of the classical Szeg\H{o}
theorem on the asymptotic behavior of OPUC \cite[Theorem
2.4.1]{Sopuc}, since we can choose
$$ \Psi_{2n}(t)=\overline{D(t)}t^{-n}p_{2n}(t), \qquad
\Psi_{2n+1}(t)=-\a_{-1}\overline{D(t)}t^{n}\overline{p_{2n+1}(t)}
$$
as a solution of \eqref{gi2}, where $p_n$ are orthonormal
polynomials with respect to $\sigma$ (cf. \cite[Lemma
4.3.14]{Sopuc}).

\subsection{Main Objectives and Results}
The main objective of this paper is solving the inverse scattering
problem (the heart of the Faddeev--Marchenko theory \cite{Mar1,
Mar2, FAD}), i.e., reconstructing the CMV matrix $\fA$ from its
scattering function $s$. In general, the solution of this inverse
problem is not unique. In particular, $s$ does not contain any
information about the (possible) singular measure.
Even in the class of absolutely continuous measures
 the correspondence $\fA\mapsto s$ is not one to one
(see Examples \ref{polweight} and \ref{jac}). In this paper we show
that the uniqueness in the inverse scattering is equivalent to the
\textit{Arov regularity} (Definition \ref{regular}) of the function
$\phi$, see Theorem~\ref{thm1} below.

We also consider two interesting subclasses of the uniqueness class,
namely, Helson--Szeg\H{o} and the Strong Szeg\H{o}. The first class
is exactly the one for which a certain {\it transformation operator}
\footnote{A classical monograph on the subject is \cite{Mar1}, where
transformation operators are extensively used in spectral and
scattering theory for Schr\" odinger operator. Historical remarks
are also given there in the introduction.} is invertible. We obtain
a complete description of the corresponding spectral measures and
the scattering functions in Section \ref{helszegoclass}. The second
class is the one for which the transformation operators have a
determinant. For this class a complete description is given to the
Verblunsky coefficients, the spectral measures and the scattering
functions in Section \ref{golibrag}.

This paper is the result of a substantial revision of the manuscript
\cite{GKPY}.

\if{Note that, already in 1979, Geronimo--Case \cite{GC0} were
discussing scattering theory for OPUC under the Baxter condition
$\sum |\a_n|<\infty$. For the role of the scattering function in
OPUC theory, particularly in extension of Baxter's theory see
\cite{GMF}.

To complete the introduction, we quote Tao and Thiele \cite{TTh}:
"Scattering transforms come in many facets, and the algebraic and
geometric part of the subject alone occupies a vast literature. But
a recent surge of papers on basic analytic properties shows that
these properties are not fully understood even for simple
models".}\fi

\section{Adamyan--Arov--Krein Theory}\label{aakthscat}

We begin with the following
\begin{definition} Pairs $(\phi, \psi)$ with properties $\phi, \psi$ are in
$H^\infty$, $\phi(0)=0$, $\psi$ is an outer function, $\psi(0)>0$,
and $|\phi|^2+|\psi|^2=1$ are called {\it $\gamma$-generating}.
\end{definition}
Recall that such pairs appear in spectral analysis of CMV matrices
(see Introduction).
\begin{proposition}\label{gener1}
To every $\gamma$-generating pair $(\phi, \psi)$ one can
associate the family of functions $($compare to \eqref{desssim}$)$
\bg{\eq}\label{0045}
s_{\cE}=-\frac{\psi}{\overline\psi}\,\frac{\cE+\overline\phi}{1+\cE\phi}\,,
\qquad \cE\in
 H^\infty, \quad \|\cE\|_\infty\le 1.
\end{\eq}
All the functions $s_{\cE}$ belong to the unit ball of $L^\infty$.
Moreover, all functions in formula \eqref{0045} have the same
negative part of the Fourier series.
\end{proposition}
\begin{proof}
 The first assertion follows from the
relation
$$
1- |s_{\cE}|^2=\frac{(1-|\cE|^2)(1-|\phi|^2)}{|1+\cE\phi|^2}.
$$
 Let $s_0$ correspond to $\cE=0$, then
\bg{\eq}\label{0046}
s_{\cE}-s_0=-\frac{\psi}{\overline\psi}\,\frac{\cE+\overline\phi}{1+\cE\phi}+
\frac{\psi}{\overline\psi}\,\overline\phi=
-\frac{\psi^2\cE}{1+\cE\phi}\in H^\infty .
\end{\eq}
\end{proof}

The following observation will be helpful later on. For each
$\gamma$-generating pair $(\phi,\psi)$ and any Schur class function
$\cE$ the function
 \bg{\eq}\label{fracschur}
 D_\cE(z):=\frac{\psi(z)}{1+\cE(z)\phi(z)}
 \end{\eq}
is an outer function from $H^2$. Indeed, $D_\cE$ is the outer
function (as a ratio of outer functions) from the Smirnov class, and
$$ |D_\cE(t)|^2=\frac{1-|\phi(t)|^2}{|1+\cE(t)\phi(t)|^{2}}\le
\frac{1-|\cE(t)\phi(t)|^2}{|1+\cE(t)\phi(t)|^{2}}=
\Re\,\frac{1-\cE(t)\phi(t)}{1+\cE(t)\phi(t)}.
$$
The right hand side is the boundary value of the Poisson integral of
a finite positive measure, and so belongs to $L^1(\bbT)$.

\smallskip

The AAK Theory deals with the following Nehari problem \cite{AAK1,
AAK2, AAK3, Garnett}.

\begin{problem}[Nehari]
Given function $h\in L^\infty,\ \Vert h \Vert_\infty\le 1$, describe
collection $\cN(h)$ of all functions
$$
\cN(h)=\{f\in L^\infty:\ \Vert f \Vert_\infty\le 1,\ f-h\in
H^\infty\},
$$
that is, the collection of functions $f\in L^\infty$ with the same
Fourier coefficients with negative indices as $h$.
\end{problem}
The Nehari problem is indeterminate (determinate) if it has
infinitely many solutions (a unique solution). It follows from
Proposition \ref{gener1} that $s$ \eqref{desssim} is a unimodular
solution of indeterminate Nehari problem.

By Proposition \ref{gener1} for every $\gamma$-generating pair
$(\phi, \psi)$ the family $\{s_\cE\}$ \eqref{0045} solves a certain
Nehari problem, generated by, e.g., $s_0$. However, formula
\eqref{0045} may not produce {\it all the functions} from the unit
ball of $L^\infty$ with {\it this} negative part of the Fourier
series.
\begin{definition}\label{regular}
A $\gamma$-generating pair $(\phi,\psi)$, or simply a function
$\phi$, are called {\it Arov-regular} (see \cite{arov}) if formula \eqref{0045}
produces {\it all the functions} from the unit ball of $L^\infty$
with a certain negative part of the Fourier series. \end{definition}
\begin{definition}
We say that a $CMV$ matrix of the Szeg\H{o} class is regular, if the
associated function $\phi$ \eqref{0001} is Arov-regular.
\end{definition}

An important result is proved in \cite[Remark 4.1]{AAK2}.
\begin{theorem}[AAK]\label{aakac}
If $\phi$ is Arov-regular, then for every Schur class function $\cE$
the measure $\sigma_\cE$
 \bg{\eq}\label{alexan}
\frac{1-\cE(z)\phi(z)}{1+\cE(z)\phi(z)}=
\int\limits_{\bbT}\frac{t+z}{t-z}\,\sigma_\cE(dt)
 \end{\eq}
is absolutely continuous.
\end{theorem}

For $h\in L^\infty$, $\|h\|_\infty\le1$, we define a Hankel operator
$\cH : H^2\to H^2_-$ as
$$
\cH=\cH_h=P_-h|H^2,
$$
$h$ is called a {\it symbol} of $\cH$.  Note that
$$ \|\cH\|\le\|h\|_\infty\le1, $$
and the adjoint operator $\cH^* : H^2_-\to H^2$ is $\cH^*=P_+ \overline
h|H^2_-$, $P_+$ ($P_-$) is the standard projection from $L^2$ onto
$H^2$ ($H^2_-$). For $\|f\|_\infty\le 1$, $\cH_f=\cH$ if and only if
$f\in\cN(h)$.

A Hankel operator $\cH_h$ is called {\it indeterminate}, if it has
many symbols $f$ with $\|f\|_\infty\le 1$.
 \begin{theorem}[Adamyan--Arov--Krein]\label{aakth}
The Nehari problem is  indeterminate if and only if
 \begin{equation}\label{aak1}
{\bf 1}\in (I-\cH^*\cH)^{1/2}\,H^2.
 \end{equation}
In this case the set $\cN(\cH)$ is of the form
\begin{equation}\label{dcrp-nehclass}
    \cN(\cH)=\{f_{\cE}=-\frac{\psi_\cH}{\ovl\psi_\cH}\frac{\cE+\ovl\phi_\cH}{1+\cE\phi_\cH}: \cE\in
 H^\infty,\ \|\cE\|_\infty\le 1\},
\end{equation}
where $(\phi_\cH, \psi_\cH)$ is a uniquely determined Arov-regular
pair, $\psi_\cH(0)>0$.
\end{theorem}

The next theorem gives sufficient conditions for regularity of
$\phi$. The second condition is known (see, e.g., \cite{arov, sarason}).
For a weaker condition on $|\psi|$, which ensures regularity of
$\phi$, see \cite{VYu2}.
\begin{theorem}\label{0018}
$\phi$ is Arov-regular as soon as one of the following conditions
holds
\begin{enumerate}
\item $\sigma_\tau$ $\eqref{alexan}$ is absolutely continuous for
some unimodular constant $\tau$, and $(1+\tau\phi)\psi^{-1}\in H^2$;
\item $\psi^{-1}\in H^2$.
\end{enumerate}
\end{theorem}
\begin{proof} (1). We consider a unimodular function
\bg{\eq}\label{0080}
s=-\frac{\psi}{\overline\psi}\,\frac{\tau+\overline\phi}{1+\tau\phi}=
-\tau\,\frac{\psi}{\overline\psi}\,\frac{1+\overline
\tau\overline\phi}{1+\tau\phi}.
\end{\eq}
We associate an indeterminate Nehari problem to $s$ with the Hankel
operator $\cH=\cH_s$. By Theorem \ref{aakth} $s$ admits the
representation
 \bg{\eq}\label{0081}
s=-\frac{\psi_\cH}{\ovl\psi_\cH}\,\frac{\cE+\ovl\phi_\cH}{1+\cE\phi_\cH}
\end{\eq}
with the Arov-regular pair $(\phi_\cH,\psi_\cH)$ and the inner
function $\cE$, so we can write
$$
s =-\cE\,\frac{\psi_\cH}{\ovl\psi_\cH}\,\frac{1+\ovl\cE\ovl\phi_\cH
}{1+\cE\phi_\cH}.
$$
Combining \eqref{0080} and \eqref{0081}, we get
\begin{equation}\label{h1h1}
    G:=\cE\,\frac{1+\tau\phi}{\psi}\,\frac{\psi_\cH}{1+\cE\phi_\cH}=
    \tau\,\overline{\frac{1+\tau\phi}{\psi}\,\frac{\psi_\cH}{1+\cE\phi_\cH}}.
\end{equation}
It was mentioned above (see \eqref{fracschur}) that
$\psi_\cH(1+\cE\phi_\cH)^{-1}\in H^2$, so, due to the assumption,
$G\in H^1$. At the same time $\ovl{G}\in H^1$, so $G$ is a constant
function. Since $\cE$ is the inner part of $G$, we have $\cE=const$.
Using the normalization $\psi(0)>0, \psi_\cH(0)>0$, we get
$\cE=\tau$ and  $\ovl\tau G>0$. Next, by \eqref{h1h1} Next
\begin{equation*}\label{0082}
   \ovl\tau G\,\frac{\psi}{1+\tau\phi}=\frac{\psi_\cH}{1+\tau\phi_\cH}\,.
\end{equation*}
so, in particular,
\begin{equation*}\label{0083}
   (\overline\tau G)^2\,\left|\frac{\psi}{1+\tau\phi}\right|^2=
   \left|\frac{\psi_\cH}{1+\tau\phi_\cH}\right|^2.
\end{equation*}
In other words,
\begin{equation*}\label{0084}
  (\overline\tau G)^2\,\Re\,\frac{1-\tau\phi}{1+\tau\phi}=
  \Re\,\frac{1-\tau\phi_\cH}{1+\tau\phi_\cH}
\end{equation*}
almost everywhere on the unit circle.

By the assumption $\sigma_\tau$ is absolutely continuous, and by
Theorem \ref{aakac} $\sigma_{\tau,\cH}$ is absolutely continuous.
Since $\phi(0)=\phi_\cH(0)=0$, $\sigma_\tau$ and $\sigma_{\tau,\cH}$
are probability measures. Hence, $\ovl\tau G=1$,
\begin{equation*}\label{0085}
  \frac{1-\tau\phi}{1+\tau\phi}=\frac{1-\tau\phi_\cH}{1+\tau\phi_\cH}\,.
\end{equation*}
Therefore $\phi=\phi_\cH$, as claimed.

(2). Let us show first that $\sigma$ is absolutely continuous.
Indeed, by \eqref{0003}
$$ \frac1{1+\bara\phi(z)}=\frac{D(z)}{\psi(z)}\in H^1\Rightarrow
R(z)=\frac{1-\bara\phi(z)}{1+\bara\phi(z)}\in H^1. $$ By the
Fihtengoltz theorem
$$ R(z)=\int\limits_{\bbT}\frac{R(t)}{1-\overline tz}\,m(dt), $$
and
$$ 1=R(0)=\int\limits_{\bbT}R(t)m(dt)=\int\limits_{\bbT}\Re
R(t)m(dt)=\int\limits_{\bbT}w(t)m(dt), $$ so $\sigma_s=0$, as
claimed. Next, by the assumption
$$ \frac1{D(z)}=\frac{1+\bara\phi(z)}{\psi(z)}\in H^2, $$
and the second statement of the theorem follows from the first one.
\end{proof}

\begin{definition}\label{0008}
If $\phi$ is Arov-regular and $\cE$ is a constant function,
$|\cE|=1$, then the function
$$ s_{\cE}=-\frac{\psi}{\overline\psi}\,\frac{\cE+\overline\phi}{1+\cE\phi} $$
is called {\it canonical}. Such $s_\cE$ is also called a {\it
canonical symbol} of the associated Hankel operator $\cH$.
\end{definition}
\begin{proposition}\cite{AAK1, AAK2, sarason, Kh0}\label{0065}
Let $s$ be a unimodular function on $\mathbb T$. Then the following
are equivalent
\begin{enumerate}
\item $ s$ is canonical,\smallskip
\item $ P_+s|H^2_+ $ is dense in $ H^2_+$,\smallskip

\noindent $ P_+t s|H^2_+ $ is not dense in $H^2_+$ $($the space is
of codimension one$)$,
\item\smallskip $\overline s h_+=h_-$
has only the trivial solution,

\smallskip\noindent $\overline s h_+=th_-$ has a
nontrivial solution $($the space of solutions is of dimension
one$)$, $h_{\pm}\in H^2_{\pm}$.
\end{enumerate}
\end{proposition}

As a simple consequence of Proposition \ref{0065} we have
\begin{proposition}\label{noncanon}
Let $s$ be canonical, and $N\not=0$ an integer. Then $st^N$ is
non-canonical.
\end{proposition}
\begin{proof}
Assume that both $s$ and $st^N$ are canonical. Then without
loss of generality we may assume that $N>0$. By the second condition
(3) the equation $\overline s h_+=th_-$ has a nontrivial solution.
Hence $\ovl{st^N}\,h_+=t^{1-N}h_-\in H^2_-$ also has a nontrivial
solution, which means that the first condition in (3) fails for
$st^N$. So $st^N$ is non-canonical, which is a contradiction.
\end{proof}

\section{Uniqueness in the inverse scattering}


We are interested in the following questions: given a unimodular
solution $s$ of an indeterminate Nehari problem, does there exist
$CMV$ matrix $\fA$ with this scattering function? Is such $\fA$
unique? The main result of the section gives complete answers on
these questions.

\begin{theorem}\label{thm1}\
\begin{enumerate}
\item Each regular $CMV$ matrix $\fA$ has absolutely continuous
spectral measure $\sigma(\fA)$,
and its scattering function $s$ is canonical.
\item Let $s$ be a canonical solution of an indeterminate Nehari
problem, then there exists a unique absolutely continuous $CMV$
matrix $\fA$ of Szeg\H o class, whose scattering function is $s$,
moreover $\fA$ is regular.
\item Let $s$ be a non-canonical unimodular solution of an indeterminate Nehari
problem, then there exist infinitely many absolutely continuous
$CMV$ matrices $\fA$ with scattering function $s$.
\end{enumerate}
\end{theorem}
\begin{proof} (1). Let $\fA$ be regular, so $\phi$ in \eqref{desssim}
is Arov-regular. By Theorem \ref{aakac}, $\sigma(\fA)$ is absolutely
continuous. By definition \ref{0008}, function $s$, defined by
\eqref{desssim} with $\cE=\bara$, is canonical.

\if{By Theorem \ref{aakth} each unimodular solution $s$ of an
indeterminate Nehari problem (canonical or not) can be written in
the form \eqref{uniindet} with Arov-regular $\phi_\cH$ and inner
function $\cE$. $s$ is canonical if and only if $\cE$ is a
unimodular constant.}\fi

(2). Since $s$ is canonical, we have
\begin{equation}\label{uniindet}
s=-\frac{\psi_\cH}{\ovl\psi_\cH}\,\frac{\cE+\ovl\phi_\cH}{1+\cE\phi_\cH}\,,
\end{equation}
where $\cE$ is a unimodular constant. Therefore, a solution of the
inverse scattering problem can be chosen as
$$
  \bara:=\cE,\quad
  D(z):=\frac{\psi_\cH(z)}{1+\cE\phi_\cH(z)}\,,\quad
 \sigma(dt):=\left|\frac{\psi_\cH}{1+\cE\phi_\cH}\right|^2\,m(dt).
  $$
Since
$$ R(z)=\int\limits_{\bbT}\frac{t+z}{t-z}\,\sigma(dt)=
   \frac{1-\bara\phi_{\cH}(z)}{1+\bara\phi_{\cH}(z)}\,, $$
the associated to $\sigma$ function $\phi=\phi_{\cH}$, so $\phi$ is
regular, as needed.

Assume that there are two absolutely continuous CMV matrices $\fA$
and $\fA'$ of Szeg\H o class with the scattering function $s$. The
corresponding spectral measures are $\sigma=|D|^2m$ and
$\sigma'=|D'|^2m$,
\begin{equation}\label{0116}
\int\limits_{\mathbb T}|D|^2\,m(dt)=\int\limits_{\mathbb
T}|D'|^2\,m(dt)=1,\quad D(0)>0,\quad D'(0)>0.
\end{equation}
Then we have
\begin{equation}\label{0115}
 s(t)=-\ovl\a_{-1}\,\frac{D(t)}{\overline{D(t)}}=-\ovl\a_{-1}'\,\frac{D'(t)}{\overline{D'(t)}}
\end{equation}
and
$$
-\ovl\a_{-1}D(t)\overline{s(t)}=\overline{D(t)} \qquad
-\ovl\a'_{-1}D'(t)\overline{s(t)}=\overline{D'(t)}.
$$
There exist two real nonzero constants $\alpha$ and $\alpha '$ such
that
$$
\alpha D(0)+\alpha' D'(0)=0,
$$
Then
$$
h_-=\overline{\alpha D+\alpha' D'}\in H^2_-,\qquad
h_+=-\ovl\a_{-1}\alpha D-\ovl\a'_{-1}\alpha' D'\in H^2_+
$$
is a solution of $\overline sh_+=h_-$. Since $s$ is canonical, by
Proposition \ref{0065}, (3), this is a trivial solution. In other
words,
$$
\alpha D+\alpha' D'=0
$$
identically. In view of \eqref{0116}, this yields $D=D'$. The
uniqueness follows.

(3). If $s$ is a non-canonical unimodular solution of an
indeterminate Nehari problem, then in \eqref{uniindet} $\cE$ is a
non-constant inner function, and \eqref{uniindet} can be rephrased
as
$$
s= \cE\,
\frac{\psi_\cH}{\ovl\psi_\cH}\,\frac{1+\overline\cE\ovl\phi_\cH}{1+\cE\phi_\cH}
= \ovl\tau\,\frac{1-\tau\cE}{1-\ovl\tau\ovl\cE}\,
\frac{\psi_\cH}{\ovl\psi_\cH}\,\frac{1+\overline\cE\ovl\phi_\cH}{1+\cE\phi_\cH}
\,, \qquad\forall\tau\in\bbT.
$$
Therefore, we get infinitely many solutions of the inverse
scattering problem
$$
  \bara=-\frac{\overline\tau-\cE(0)}{1-\overline\tau\overline {\cE(0)}}\,,\qquad
  D(z)=k_\tau\,\frac{|1-\tau\cE(0)|}{1-\tau\cE(0)}\,
  \frac{(1-\tau\cE(z))\psi_\cH(z)}{1+\cE(z)\phi_\cH(z)}\ \in H^2,
  $$
  where $k_\tau>0$ is chosen to make $\int\limits_{\mathbb T}|D(t)|^2m(dt)=1$.
It is verified by a straightforward computation that indeed $\bara$
and $|D|$ are different for different $\tau$.
\end{proof}
\begin{corollary}\label{aakscatter}\
\begin{enumerate}
\item Let $\fA$ be a regular $CMV$ matrix, let $\fA_1$ be an absolutely
continuous $CMV$ matrix of the Szeg\H o class. If they have the same
scattering function $s$ then
    $\fA_1=\fA$.
\item Let $\fA$ be a non-regular absolutely continuous $CMV$ matrix of
the Szeg\H o class with the scattering function $s$. Then there
exist infinitely many absolutely continuous $CMV$ matrices of the
Szeg\H o class with the same scattering function.
\end{enumerate}
\end{corollary}

\begin{remark}
As we saw earlier, for every CMV matrix of Szeg\H o class, its
scattering function is a unimodular solution of an indeterminate
Nehari problem. As a byproduct of this section, we have shown that
every unimodular solution of an indeterminate Nehari problem is the
scattering function of an absolutely continuous $CMV$ matrix $\fA$.
\end{remark}

We complete with a simple example, when the solution of the inverse
scattering problem is not unique.

\begin{example}\label{polweight}
Let
$$ P(z)=\prod_{j=1}^N (z-t_j), \qquad t_j\in\bbT $$
be a monic polynomial of degree $N$ with all zeros on $\bbT$. For
the measure
$$ \sigma(dt)=w(t)m(dt), \quad w(t):= c|P(t)|^2=c\prod_{j=1}^N |t-t_j|^2, \quad c>0, $$
the Szeg\H{o} function $D=\sqrt{c}P/P(0)$, and the scattering
function is
$$ s(t)=-\bara\,\frac{D(t)}{\overline{D(t)}}=-\bara\overline{P(0)}t^N. $$
Thus for any two polynomials $P_1,P_2$ with $P_1(0)=P_2(0)$ we have
$s_1=s_2$, and there is no uniqueness in the inverse scattering even
for $\a_{-1}=-1$. Note that $s$ is not canonical.

In the case $N=1$ we have $s=\bara\overline t_1\,t$, and again there is
no uniqueness.
\end{example}

\section{Schur algorithm}

It is convenient to deal with two sequences $\{f_n\}_{n\ge0}$ and
$\{\phi_n\}_{n\ge0}$ given for $n=0,1,\ldots$ by
\begin{equation}\label{schuralg}
\begin{split}
f_{n+1}(z) &=\frac{f_n(z)-f_n(0)} {z(1-f_n(z)\ovl{f_n(0)})}\,, \quad
zf_0(z)=\phi_0(z)=\phi(z); \\  \phi_n(z) &=zf_n(z)
\end{split}
\end{equation}
from the Schur class. By the Geronimus theorem
\begin{equation}\label{geron}
f_n(0)=a_n:=-\a_{-1}\,\a_n,\qquad n=0,1\ldots. \end{equation}

If $\phi$ is a Szeg\H o function, then all the functions
$\phi_n$ \eqref{schuralg} are also Szeg\H o functions.
So, we can define a sequence
of $\gamma$-generating pairs $(\phi_n,\psi_n)$. It is easy to see
that $\{\psi_n\}$ satisfies
\begin{equation}\label{schuralgpsi}
\psi_{k+1}=\psi_k\frac{\rho_k}{1-\overline a_k f_k}\,, \quad
\psi_n=\psi\,\prod_{k=0}^{n-1}\frac{\rho_k}{1-\overline a_kf_k}\,,
\quad \psi=\psi_0.
\end{equation}
Indeed, for $t\in\bbT$
$$
|\psi_{k+1}(t)|^2=1-|f_{k+1}(t)|^2=
\frac{(1-|f_k(t)|^2)(1-|\a_k|^2)}{|1-\ovl a_kf_k(t)|^2}=
\frac{|\psi_k(t)|^2\,\rho_k^2}{|1-\ovl a_kf_k(t)|^2}\,. $$ It is
also clear from \eqref{0003} and \eqref{geron} that
\begin{equation}\label{psiorig}
\psi_n(0)=\psi(0)\,\prod_{k=0}^{n-1}\rho_k^{-1}=\prod_{k=n}^{\infty}\rho_k.
\end{equation}

\begin{lemma}\label{0096}
Recurrences \eqref{schuralg} and \eqref{schuralgpsi} can be put into
the form \bg{\eq}\label{0097}
\begin{bmatrix}
\frac{1}{\overline\psi_n} & \frac{\overline\phi_n}{\overline\psi_n}\\ \\
\frac{\phi_n}{\psi_n} & \frac{1}{\psi_n}
\end{bmatrix}
=
\begin{bmatrix} \overline t & 0 \\ 0 & 1\end{bmatrix}
\begin{bmatrix}
\frac{1}{\overline\psi_{n+1}} & \frac{\overline\phi_{n+1}}{\overline\psi_{n+1}}\\ \\
\frac{\phi_{n+1}}{\psi_{n+1}} & \frac{1}{\psi_{n+1}}
\end{bmatrix}
\begin{bmatrix} t & \overline a_n \\ a_n t & 1\end{bmatrix}
\frac{1}{\rho_n}\,.
\end{\eq}
\end{lemma}
\begin{proof}
By \eqref{schuralg}, $$(1-\ovl a_nf_n)\phi_{n+1}=f_n-a_n,$$
and
$$
(1-\ovl a_nf_n)(1+\ovl a_n\phi_{n+1})=1-|a_n|^2=\rho_n^2.
$$
Therefore,
$$ \frac{\rho_n^2}{1-\ovl a_nf_n}=1+\ovl a_n\phi_{n+1}. $$
Next, by \eqref{schuralgpsi},
$$ \frac{1}{\psi_n}
=\frac{1}{\psi_{n+1}}\frac{\rho_n}{1-\ovl a_nf_n}
=\frac{1+\ovl a_n\phi_{n+1}}{\psi_{n+1}\rho_n}
=\left(\frac{\phi_{n+1}}{\psi_{n+1}}\ \ovl a_n
+\frac{1}{\psi_{n+1}}\right)\frac{1}{\rho_n},
$$
which is
$(2,2)$ entry of \eqref{0097}.

Similarly, by \eqref{schuralg},
\begin{equation*}
(1-\ovl a_nf_n)(\phi_{n+1}+a_n) =\ovl{t}\rho_n^2\phi_n,
\end{equation*}
and
\begin{equation*}
\frac{\rho_n^2\phi_n}{1-\ovl a_nf_n} =t(\phi_{n+1}+a_n).
\end{equation*}
Therefore, by \eqref{schuralgpsi},
$$ \frac{\phi_n}{\psi_n}
=\frac{1}{\psi_{n+1}}\frac{\rho_n\phi_n}{1-\ovl a_nf_n}
=\frac{t(\phi_{n+1}+a_n)}{\psi_{n+1}\rho_n}
=t\left(\frac{\phi_{n+1}}{\psi_{n+1}}
+\frac{1}{\psi_{n+1}}a_n\right)\frac{1}{\rho_n},
$$
which is $(2,1)$ entry of \eqref{0097}.
\end{proof}

Repeatedly applying \eqref{0097} we get for $n>j$
\bg{\eq}\label{0110}
\begin{bmatrix}
\frac{1}{\overline\psi_j} & \frac{\overline\phi_j}{\overline\psi_j}\\ \\
\frac{\phi_j}{\psi_j} & \frac{1}{\psi_j}
\end{bmatrix}
=
\begin{bmatrix} \overline t^{(n-j)} & 0 \\ 0 & 1\end{bmatrix}
\begin{bmatrix}
\frac{1}{\overline\psi_{n}} & \frac{\overline\phi_{n}}{\overline\psi_{n}}\\ \\
\frac{\phi_{n}}{\psi_{n}} & \frac{1}{\psi_{n}}
\end{bmatrix}
\left( \overset{\longleftarrow} {\prod\limits_{k=j}^{n-1}}
\begin{bmatrix} t & \ovl a_k \\ a_k t & 1\end{bmatrix}
\frac{1}{\rho_k} \right).
\end{\eq}
We define
\bg{\eq}\label{0114}
\begin{bmatrix}
\cP_j^j
\\ \\
\cQ_j^j
\end{bmatrix}
=
\begin{bmatrix}
0
\\ \\
1
\end{bmatrix}
\end{\eq}
and for $n>j$
\bg{\eq}\label{0111}
\begin{bmatrix}
\cP_n^j(z)
\\ \\
\cQ_n^j(z)
\end{bmatrix}
= \left( \overset{\longleftarrow} {\prod\limits_{k=j}^{n-1}}
\begin{bmatrix} z & \ovl a_k \\ a_k z & 1\end{bmatrix}
\frac{1}{\rho_k} \right)
\begin{bmatrix}
0
\\ \\
1
\end{bmatrix}.
\end{\eq}
Note that $\cP_n^j$ and $\cQ_n^j$ are polynomials,
\begin{equation}\label{polin}
{\rm deg}\,\cP_n^j\le n-j-1, \quad {\rm deg}\,\cQ_n^j\le n-j-1,
\quad \cQ_n^j(0)=\prod_{k=j}^{n-1} \rho_k^{-1}>0.
\end{equation}
It is easily seen from \eqref{0111} that
$$ \begin{bmatrix}
\cP_n^0 & \cP_n^j \\ \cQ_n^0 & \cQ_n^j\end{bmatrix}= \left(
\overset{\longleftarrow} {\prod\limits_{k=j}^{n-1}}
\begin{bmatrix} t & \ovl a_k \\ a_k t & 1\end{bmatrix}
\frac{1}{\rho_k} \right)
\begin{bmatrix}
\cP_j^0 & 0 \\ \cQ_j^0 & 1\end{bmatrix}. $$
 Taking determinants we come to
\begin{equation}\label{deter}
\cP_n^0(z)\cQ_n^j(z)-\cQ_n^0(z)\cP_n^j(z)=z^{n-j}\cP_j^0(z).
\end{equation}
From \eqref{0110} and \eqref{0111} we have
\begin{equation}\label{01100111}
\frac{\ovl\phi_j}{\ovl\psi_j}=t^{j-n}\,\frac{\cP_n^j+\ovl\phi_n\cQ_n^j}{\ovl\psi_n}\,,
\quad \frac{1}{\psi_j}=\frac{\cP_n^j\phi_n+\cQ_n^j}{\psi_n}\,.
\end{equation}
\begin{remark}
Matrix products \eqref{0111} arise in the Szeg\H{o} recurrences for
OPUC (see \cite[formula (1.5.35)]{Sopuc}).
\end{remark}

We also define \bg{\eq}\label{0112}
\cE_n^j=\frac{\cP_n^j}{\cQ_n^j}\,, \qquad n\ge j.
\end{\eq}
It is clear from \eqref{0111} that $\cE_n^j$ can be defined
recursively as
\bg{\eq}\label{0101} \cE_j^j=0, \quad
\cE_{n+1}^j=\frac{t\cE_n^j+\overline a_n}{1+a_n\,t\cE_n^j},\quad
n\ge j,
\end{\eq}
so $\|\cE_n^j\|_\infty<1$ for $n\ge j$.
\begin{remark} Using those notations we can rewrite \eqref{deter} as
\begin{equation}\label{ordzero}
\cE_n^0(z)-\cE_n^j(z)=\frac{\cP_n^0(z)\cQ_n^j(z)-\cQ_n^0(z)\cP_n^j(z)}{\cQ_n^0(z)\cQ_n^j(z)}
=\frac{z^{n-j}\cP_j^0(z)}{\cQ_n^0(z)\cQ_n^j(z)}\,,
\end{equation}
which implies, in view of \eqref{polin}, that
the difference
\begin{equation}\label{0117}
\cE_n^0(z)-\cE_n^j(z)
\end{equation}
vanishes at the origin with order of at least $n-j$.

The second equality in \eqref{01100111} also can be rewritten as
$$
\frac{\psi_n}{\psi_j}=\cP_n^j\phi_n+\cQ_n^j=\cQ_n^j(1+\cE_n^j\phi_n).
$$
Hence,
\begin{equation}\label{psice}
\frac{\psi_n}{\psi_j}\,\frac1{1+\cE_n^j\phi_n}
=\cQ_n^j
\end{equation}
and, therefore, is a polynomial of degree at most $n-j-1$.
\end{remark}
\begin{lemma}\label{0102}
Let $s_k:=-\frac{\psi_k}{\overline\psi_k}\,\overline\phi_k$. Then,
for $n\ge j$
\bg{\eq}\label{0100}
 s_j=-\overline t^{(n-j)}
 \frac{\psi_n}{\overline\psi_n}\,\frac{\cE_n^j+\overline\phi_n}{1+\cE_n^j\phi_n}
\end{\eq}
and $($cf. $\eqref{0046})$
 \bg{\eq}\label{0103}
s_{j}t^{n-j}-s_n=-\frac{\psi_n^2\cE_n^j}{1+\cE_n^j\phi_n} \in
H^\infty,
\end{\eq}
where $\cE_n^j$ are defined as in \eqref{0111}--\eqref{0112} or
$($equivalently$)$ by \eqref{0101}.
\end{lemma}
\begin{proof} Apply \eqref{0110} to $\begin{bmatrix} 0 \\ 1 \end{bmatrix}$.
\end{proof}

\section{Model space and transformation operator}

Let $\fA\in\Sz$, $(\phi,\psi)$ be the corresponding
$\gamma$-generating pair.
\begin{definition}
We define the {\it Faddeev--Marchenko space} $M_\phi$ as the Hilbert
space of analytic vector-functions
$$
\begin{bmatrix}
F_1\\ F_2
\end{bmatrix},\quad
F_1={F_+}/{\psi},\quad F_2={F_-}/{\overline\psi}, \quad F_{\pm}\in
H^2_{\pm}
$$
with the inner product
 \bg{\eq}\label{0056} \left\langle
\begin{bmatrix}
{F_1} \\ {F_2}
\end{bmatrix},
\begin{bmatrix}
{G_1} \\ {G_2}
\end{bmatrix}
\right\rangle_{M_\phi} = \int\limits_{\bbT}
\begin{bmatrix}
\overline{G_1} & \overline{G_2}
\end{bmatrix}
\begin{bmatrix}
1 & \overline s_0 \\
s_0 & 1
\end{bmatrix}
\begin{bmatrix}
F_1 \\ F_2
\end{bmatrix}
{m(dt)},
\end{\eq}
where
$
s_0=-\frac{\psi}{\overline\psi}\overline{\phi}.
$
\end{definition}
We mention that  $M_\phi$ comes up as a functional model space for
the CMV matrix $\fA$. More specifically, we can start with the
de Branges--Rovnyak model space $K_\phi$: $F_{\pm}\in H^2_{\pm}$,
\begin{eqnarray*}
\left\Vert
\begin{bmatrix}
{F_+} \\ {F_-}
\end{bmatrix}
\right\Vert^2_{K_\phi}&=&
\int\limits_{\bbT}
\begin{bmatrix}
\overline{F_+(t)} & \overline{F_-(t)}
\end{bmatrix}
\begin{bmatrix}
1 & \phi(t) \\
\overline{\phi(t)} & 1
\end{bmatrix}^{[-1]}
\begin{bmatrix}
{F_+(t)} \\ {F_-(t)}
\end{bmatrix}
m(dt)
\nonumber
\end{eqnarray*}
and transform it as follows
\begin{eqnarray*}
\qquad &=&
\int\limits_{\bbT}
\begin{bmatrix}
\overline{F_+(t)} & \overline{F_-(t)}
\end{bmatrix}
\begin{bmatrix}
1 & -\phi(t) \\
-\overline{\phi(t)} & 1
\end{bmatrix}
\begin{bmatrix}
{F_+(t)} \\ {F_-(t)}
\end{bmatrix}
\frac{m(dt)}{1-|\phi(t)|^2}
\nonumber\\
\qquad &=&
\int\limits_{\bbT}
\begin{bmatrix}
\overline{F_+(t)} & \overline{F_-(t)}
\end{bmatrix}
\begin{bmatrix}
1 & -\phi(t) \\
-\overline{\phi(t)} & 1
\end{bmatrix}
\begin{bmatrix}
{F_+(t)} \\ {F_-(t)}
\end{bmatrix}
\frac{m(dt)}{|\psi(t)|^2}
\nonumber\\
\qquad &=&
\int\limits_{\bbT}
\begin{bmatrix}
\overline{F_+/\psi} & \overline{F_-/\overline\psi}
\end{bmatrix}
\begin{bmatrix}
1 & \overline s_0 \\
s_0 & 1
\end{bmatrix}
\begin{bmatrix}
{F_+}/{\psi} \\  \\ {F_-}/{\overline\psi}
\end{bmatrix}
{m(dt)}.
\end{eqnarray*}
\begin{proposition}\label{0118}
Linear manifold $\begin{bmatrix} H^2_+
\\ \\ H^2_-\end{bmatrix}$ is contained in $M_\phi$, and
\bg{\eq}\label{0057}
 \left\langle
\begin{bmatrix}
{h_+} \\ {h_-}
\end{bmatrix},
\begin{bmatrix}
{g_+} \\ {g_-}
\end{bmatrix} \right\rangle_{M_\phi} =
\left\langle\begin{bmatrix}
I & \cH^* \\
\cH & I
\end{bmatrix}
\begin{bmatrix}
{h_+} \\ {h_-}
\end{bmatrix},
\begin{bmatrix}
{g_+} \\ {g_-}
\end{bmatrix}
\right\rangle_{L^2}.
\end{\eq}
\end{proposition}
\begin{proof}
Let $h_{\pm}\in H^2_{\pm}$. Then
$$
h_+=\frac{\psi h_+}{\psi},\quad h_-=\frac{\overline\psi h_-}{\overline\psi}
$$
and
\begin{equation*}
\begin{split}
\left\|\begin{bmatrix} {h_+} \\ {h_-}
\end{bmatrix}\right\|^2_{M_\phi} &=\int\limits_{\bbT}
(|h_+|^2+|h_-|^2+s_0h_+\ovl h_- +\ovl{s_0h_+}h_-)\,m(dt)\\
&= \|h_+\|^2+\|h_-\|^2+\langle s_0h_+,h_-\rangle+\ovl{\langle
s_0h_+,h_-\rangle}.
\end{split}
\end{equation*}
But $\langle s_0h_+,h_-\rangle=\langle P_-s_0h_+,h_-\rangle= \langle
\cH h_+,h_-\rangle$, so
$$ \left\|\begin{bmatrix} {h_+} \\ {h_-}
\end{bmatrix}\right\|^2_{M_\phi}=\|h_+\|^2+\|h_-\|^2+\langle
\cH h_+,h_-\rangle+\ovl{\langle \cH h_+,h_-\rangle}, $$ as claimed.
\end{proof}
The next theorem was proved in \cite{Kh1, Kh0, sarason}, see also
\cite{BallKhe}.
\begin{theorem} \label{0137}
$\phi$ is Arov--regular if and only if the set
$\begin{bmatrix} H^2_+ \\ \\ H^2_-\end{bmatrix}$ is dense
in~$M_\phi$.
\end{theorem}
\if{ 
An important linear subspace is
$$ \cH^+:=\left\{\begin{bmatrix}
{h_+} \\ {-\cH h_+}\end{bmatrix}, \quad h_+\in H^2_+\right\}\subset
\begin{bmatrix} H^2_+ \\ H^2_-\end{bmatrix}.
$$
By \eqref{0057}
\begin{equation}\label{inprod}
 \left\langle
\begin{bmatrix}
{h_+} \\ {-\cH h_+}
\end{bmatrix},
\begin{bmatrix}
{g_+} \\ {-\cH g_+}
\end{bmatrix} \right\rangle_{M_\phi} =
\langle (I-\cH^*\cH)h_+, g_+\rangle_{H^2}.
\end{equation}
\smallskip
}\fi 
Let $\fA\in\Sz$, $(\phi_n,\psi_n)$ be the sequence of
$\gamma$-generating pairs related to the Schur algorithm.

\begin{lemma}\label{0066}
The vectors
\bg{\eq}\label{0059}
\ff_n=\begin{bmatrix}\frac{t^n}{\psi_n}\\ \\
\frac{\overline\phi_n}{\overline\psi_n}\end{bmatrix}
\end{\eq}
form an orthonormal system in the Faddeev--Marchenko space $M_\phi$.
Let $M_{\phi,+}$ be the subspace in $M_\phi$ spanned by those
vectors, then $M_{\phi,+}^\perp$ consists of functions with $F_1=0$,
$F_2\in H^2_-$.
\end{lemma}
\begin{proof}
Due to recurrence \eqref{schuralgpsi},
$$
\frac{t^n}{\psi_n}=\frac{t^nh_n}{\psi},\quad
\frac{\overline\phi_n}{\overline\psi_n}=\frac{\ovl\phi_n\ovl
h_n}{\ovl\psi}, \quad h_n\in H^\infty.
$$
Using Lemma \ref{0102}, we first compute
\begin{equation*}
\begin{split}
\begin{bmatrix}
1 & \overline s_0 \\
s_0 & 1
\end{bmatrix}
\begin{bmatrix}\frac{t^n}{\psi_n}\\ \\
\frac{\overline\phi_n}{\overline\psi_n}\end{bmatrix} &=
\begin{bmatrix}
\overline s_0\frac{\overline\phi_n}{\overline\psi_n} +
\frac{t^n}{\psi_n}\\ \\
\frac{s_0 t^n}{\psi_n}
+\frac{\overline\phi_n}{\overline\psi_n}\end{bmatrix} =
\begin{bmatrix}
t^n\left(\overline s_n\frac{\overline\phi_n}{\overline\psi_n}
-\frac{\overline{\phi_n\psi_n\cE_n}}{1+\ovl{\cE_n\phi_n}} +
\frac{1}{\psi_n}\right)\\ \\
\frac{s_n}{\psi_n} -\frac{\psi_n\cE_n}{1+\cE_n\phi_n}
+\frac{\overline\phi_n}{\overline\psi_n}\end{bmatrix} \\
&=
\begin{bmatrix}
t^n \frac{\overline{\psi_n}}{1+\ovl{\cE_n\phi_n}}
\\ \\
-\frac{\cE_n\psi_n}{1+\cE_n\phi_n}
\end{bmatrix}, \qquad \cE_n=\cE_n^0.
\end{split}
\end{equation*}
Next, we assume that $m\ge n$ and compute
\begin{equation}\label{0104}
\begin{bmatrix}\frac{\overline t^m}{\overline\psi_m} &
\frac{\phi_m}{\psi_m}\end{bmatrix}
\begin{bmatrix}
1 & \overline s_0 \\
s_0 & 1
\end{bmatrix}
\begin{bmatrix}\frac{t^n}{\psi_n}\\ \\
\frac{\overline\phi_n}{\overline\psi_n}\end{bmatrix}
=
\overline t^{(m-n)}\,\frac{\overline\psi_n}{\overline\psi_m}\,
\frac{1}{1+\ovl{\cE_n\phi_n}} - \frac{\psi_n}{\psi_m}\,
\frac{\cE_n\phi_m}{1+\cE_n\phi_n}\,.
\end{equation}
Since $\|\cE_n\|_\infty<1$ and due to \eqref{schuralgpsi}
$$ \frac{1}{1+\cE_n\phi_n}\in H^\infty, \qquad \frac{\psi_n}{\psi_m}\in H^\infty.  $$
Hence \eqref{0104} belongs to $L^\infty$, in particular, $\ff_n\in
M_\phi$. Now \eqref{0104} implies
$$ \langle \ff_n,\ff_m \rangle_{M_\phi}=
\int\limits_{\bbT} \overline
t^{(m-n)}\,\frac{\overline\psi_n}{\overline\psi_m}\,
\frac{1}{1+\ovl{\cE_n\phi_n}}\,m(dt)-\int\limits_{\bbT}
\frac{\psi_n}{\psi_m}\,
\frac{\cE_n\phi_m}{1+\cE_n\phi_n}\,m(dt)=\delta_{mn}. $$ The first
assertion follows.

To verify the second assertion, assume that vector
$\begin{bmatrix}F_+/\psi\\ \\F_-/\ovl\psi\end{bmatrix}$ is
orthogonal to $\ff_n$ for all $n=0,1,\ldots$. As above in
\eqref{0104}
\begin{equation}\label{0105}
\begin{bmatrix}\frac{\overline F_+}{\overline\psi} &
\frac{\overline F_-}{\psi}\end{bmatrix}
\begin{bmatrix}
1 & \overline s_0 \\
s_0 & 1
\end{bmatrix}
\begin{bmatrix}\frac{t^n}{\psi_n}\\ \\
\frac{\overline\phi_n}{\overline\psi_n}\end{bmatrix} =\overline F_+
t^n\frac{\overline\psi_n}{\overline\psi}\,
\frac{1}{1+\ovl{\cE_n\phi_n}} - \frac{\psi_n}{\psi}\,
\frac{\overline F_-\cE_n}{1+\cE_n\phi_n}\,,
\end{equation}
so
\begin{equation}\label{mphiplus}
\begin{split}
0 &=\left\langle \ff_n, \begin{bmatrix}\frac{F_+}{\psi} \\ \\
\frac{F_-}{\overline \psi}\end{bmatrix}\right\rangle_{M_\phi} =
\int\limits_{\bbT} \overline
F_+\,\frac{\overline\psi_n}{\overline\psi}\,
\frac{1}{1+\ovl{\cE_n\phi_n}}t^n\,m(dt) \\
&-\int\limits_{\bbT} \frac{\psi_n}{\psi}\, \frac{\overline
F_-\cE_n}{1+\cE_n\phi_n}\,m(dt).
\end{split}\end{equation}
The second term in the right hand side of \eqref{mphiplus} is zero,
since $F_-\in H^2_-$, so
\begin{equation}\label{fourcoef}
\int\limits_{\bbT} F_+\,\frac{\psi_n}{\psi}\,
\frac{1}{1+\cE_n\phi_n}t^{-n}\,m(dt)=0, \qquad n=0,1,\ldots.
\end{equation}
If for the contrary
$$ F_+(t)=\sum_{j\ge q} (F_+)_j\,t^j, \quad (F_+)_q\not=0, $$
$(F_+)_j$ is the $j$-th Fourier coefficient of $F_+$, then from
\eqref{fourcoef} with $n=q$
$$ \int\limits_{\bbT} (F_+)_q\,\frac{\psi_n}{\psi}\,
\frac{1}{1+\cE_n\phi_n}\,m(dt)=0 \Rightarrow
(F_+)_q\,\frac{\psi_n(0)}{\psi(0)}=0. $$ The contradiction shows
that $F_+=0$, and
$$ \int\limits_{\bbT} \begin{bmatrix} 0 &
\overline F_2\end{bmatrix}\,\begin{bmatrix}
1 & \overline s_0 \\
s_0 & 1
\end{bmatrix}\,\begin{bmatrix} 0\\ \\
F_2\end{bmatrix}\,m(dt)=\int\limits_{\bbT} |F_2(t)|^2\,m(dt)<\infty.
$$
Since $F_2$ is of the form $F_2={F_-}/{\overline\psi}$, $F_-\in
H^2_-$, $\psi$ is outer, then, by Smirnov maximum principle, $F_2\in
H^2_-$. The proof is complete.
\end{proof}
\begin{corollary}\label{0067}
$ \left\{\begin{bmatrix} {h_+} \\ {-\cH h_+}\end{bmatrix}, \quad
h_+\in H^2_+\right\}\subset M_{\phi,+}$.
\end{corollary}
\begin{proof}
By Proposition \ref{0118} the manifold
$\begin{bmatrix} H^2_+
\\ H^2_-\end{bmatrix}$ is contained in $M_\phi$.
By \eqref{0057}, for all $F_2\in H^2_-$
$$ \left\langle \begin{bmatrix} h_+ \\ \\ -\cH h_+ \end{bmatrix},
\begin{bmatrix} 0 \\ \\
F_2\end{bmatrix}\right\rangle_{M_\phi} = \left\langle
\begin{bmatrix} (I-\cH^*\cH)h_+ \\ \\ 0 \end{bmatrix},
\begin{bmatrix} 0 \\ \\
F_2\end{bmatrix}\right\rangle_{L^2}=0, $$ and the result follows
from the second assertion of Lemma \ref{0066}.
\end{proof}
\begin{definition}\label{0068}
We define a unitary operator $\widetilde\cL$ from $M_{\phi,+}$ onto
$H^2$ as
\begin{\eq}\label{0061}
\widetilde\cL\ff_{n}=
t^n.
\end{\eq}
The transformation $\cL: H^2\to H^2$ is defined as
\begin{\eq}\label{0062}
\cL h_+ = \widetilde\cL\begin{bmatrix} h_+ \\ -\cH h_+\end{bmatrix}.
\end{\eq}
$\cL$ is called the {\it transformation operator} associated to the
given sequence of Verblunsky coefficients.
\end{definition}
\begin{proposition}\label{0069} The following equality
holds true
\begin{\eq}\label{0063}
I-\cH^*\cH=\cL^*\cL.
\end{\eq}
\end{proposition}
\begin{proof}
This follows from the unitarity of $\widetilde\cL$
$$
\Vert \cL h_+ \Vert_{H^2}^2 =
\Vert \widetilde\cL\begin{bmatrix} h_+ \\ -\cH h_+\end{bmatrix}\Vert_{H^2}^2=
\Vert \begin{bmatrix} h_+ \\ -\cH h_+\end{bmatrix}\Vert_{M_{\phi}}^2=
\langle (I-\cH^*\cH)h_+, h_+\rangle.
$$
\end{proof}
Equality \eqref{0063} is called the {\it Gelfand--Levitan--Marchenko} (GLM) {\it equation}.
\begin{remark}
Similar to Lemma \ref{0066} we can show that the system of vectors
\begin{\eq}\label{0143}
\fe_{2n}=
\begin{bmatrix}
{t^n}\frac{1}{\psi_{2n}}\\ \\
\overline t^n\frac{\overline{\phi_{2n}}}{\overline{\psi_{2n}}}
\end{bmatrix},
\quad
\fe_{2n+1}=
\begin{bmatrix}
{t^n}\frac{\phi_{2n+1}}{\psi_{2n+1}}\\ \\
\overline t^{n+1}\frac{1}{\overline{\psi_{2n+1}}}
\end{bmatrix},
\quad n\ge 0
\end{\eq}
forms an orthonormal basis for $M_\phi$. Similar to Definition \ref{0068}
we can define transformation $\widetilde\cM$
\begin{\eq}\label{0139}
\widetilde\cM\fe_{2n}=
\begin{bmatrix} t^n \\ 0 \end{bmatrix},\quad
\widetilde\cM\fe_{2n+1}=
\begin{bmatrix} 0 \\ \overline t^{n+1} \end{bmatrix}.
\end{\eq}
$\widetilde\cM$ transforms the basis \eqref{0143}, associated to the given CMV
matrix $\fA$, into the basis associated to the simplest CMV matrix
(the one with $\phi=0$, $\alpha_{-1}=-1$). Operator $\widetilde\cM$ is called the {\it transformation
operator associated to the CMV matrix} $\fA$.

The transformation
$$
\cM: \begin{bmatrix} H^2_+
\\ \\ H^2_-\end{bmatrix}\to\begin{bmatrix} H^2_+
\\ \\ H^2_-\end{bmatrix}
$$
is defined as a restriction of $\widetilde\cM$
\begin{\eq}\label{0140}
\cM=\widetilde\cM\left| \begin{bmatrix} H^2_+
\\ \\ H^2_-\end{bmatrix}\right.
\end{\eq}
Similar to \eqref{0063} we can get
$$
\begin{bmatrix}
I & \cH^* \\
\cH & I
\end{bmatrix}=\cM^*\cM .
$$
However, it is more convenient for our purposes to use the operator
$\widetilde\cL$ rather than $\widetilde\cM$.
\end{remark}
\begin{proposition}\label{0070}
$\cL$ is a contraction. Matrix of $\cL$ with respect to the basis
$\{t^k\}_{k\ge 0}$
$$ \cL=\|\cL_{nm}\|_{n,m\ge0}, \qquad \cL_{nm}=\langle \cL t^m,
t^n\rangle $$ is lower triangular.
\end{proposition}
\begin{proof}
The first assertion is straightforward from \eqref{0063}. For the
second one we show that $\begin{bmatrix} t^n \\ -\cH
t^n\end{bmatrix}$ is in the span of $\{\ff_{k}\}_{k\ge n}$. Indeed,
by using the formulae of Lemma \ref{0066} we get the following
expression for the entries of $\cL$
\begin{equation*}
\begin{split}
\cL_{nm}
&=\langle \cL t^m, t^n\rangle
=\left\langle \widetilde\cL\,\begin{bmatrix} t^m \\ \\
-\cH t^m \end{bmatrix}, \widetilde\cL \ff_n\right\rangle
=\left\langle \begin{bmatrix} t^m \\ \\
-\cH t^m \end{bmatrix}, \ff_n\right\rangle_{M_\phi,+} \\
&=
\left\langle \begin{bmatrix} t^m \\ \\
-\cH t^m \end{bmatrix},
\begin{bmatrix}
1 & \overline s_0 \\
s_0 & 1
\end{bmatrix}
\ff_n\right\rangle_{L^2}
=\left\langle\begin{bmatrix} t^m \\ -\cH t^m\end{bmatrix},
\begin{bmatrix}
t^n
\frac{\overline{\psi_n}}{\overline{1+\phi_n\cE_n}}
\\ \\
-\frac{\psi_n\cE_n}{1+\phi_n\cE_n}
\end{bmatrix}
\right\rangle_{L^2}
\\
&=
\langle t^m ,
t^n
\frac{\overline{\psi_n}}{\overline{1+\phi_n\cE_n}}
\rangle_{L^2}
+\langle \cH t^m ,
\frac{\psi_n\cE_n}{1+\phi_n\cE_n}
\rangle_{L^2}
\end{split}\end{equation*}
The last term is zero, so finally
\begin{equation}\label{entrl}
\cL_{nm}=\langle \cL t^m,
t^n\rangle
=
\langle
\frac{{\psi_n}}{{1+\phi_n\cE_n}},
t^{n-m} \rangle_{L^2}
=\left(\frac{\psi_n}{1+\cE_n\phi_n}\right)_{n-m}\,.
\end{equation}
The latter is zero as long as $m>n$, as claimed.
\end{proof}
Since $\cL_{nn}=\psi_n(0)=\prod_{k=n}^\infty \rho_k>0$, all diagonal
entries of $\cL$ are nonzero numbers. Therefore, the matrix of $\cL$
has a formal inverse $\cL^{-1}=\|\cL^{-1}_{nm}\|$.
\begin{theorem}\label{0073}
The entries of the $m$-th column of the matrix $\cL^{-1}$ are
 the Taylor coefficients of the function $\frac{t^m}{\psi_m}$
\begin{equation}\label{entrinvl}
\cL^{-1}_{n,m}=\left(\frac{t^m}{\psi_m}\right)_n
=\left(\frac1{\psi_m}\right)_{n-m}.
\end{equation}
\end{theorem}
\begin{proof} Since a product of the lower triangular matrices is a
lower triangular one, need to show that for $n\ge j$
\bg{\eq}\label{0132}
\sum_{m=j}^n\cL_{nm}\,\left(\frac1{\psi_j}\right)_{m-j}=
\delta_{nj}.
\end{\eq}
In view of \eqref{entrl}
\begin{equation}\label{0107}
\sum_{m=j}^n\cL_{nm}\,\left(\frac1{\psi_j}\right)_{m-j}=
\left(\frac{\psi_n}{\psi_j}\,\frac1{1+\cE_n\phi_n}\right)_{n-j}.
\end{equation}
For $n=j$ \eqref{0132} is straightforward from \eqref{0107}. For $n>j$ we turn to
\eqref{ordzero} and write
\begin{equation*}
\begin{split}
\frac{\psi_n}{\psi_j}\,\frac1{1+\cE_n^j\phi_n} &-
\frac{\psi_n}{\psi_j}\,\frac1{1+\cE_n\phi_n} =
\frac{\psi_n}{\psi_j}\,\frac{\phi_n(\cE_n-\cE_n^j)}{(1+\cE_n^j\phi_n)(1+\cE_n\phi_n)}
\\ &=\frac{\psi_n\phi_n}{\psi_j}\,\frac{z^{n-j}\cP_j}
{(1+\cE_n^j\phi_n)(1+\cE_n\phi_n)\cQ_n\cQ_n^j}=O(z^{n-j+1}), \quad
z\to 0.
\end{split}
\end{equation*}
$+1$ comes from $\phi_n$ since $\phi_n(0)=0$.
Hence
$$
\left(\frac{\psi_n}{\psi_j}\,\frac1{1+\cE_n\phi_n}\right)_{n-j}=
\left(\frac{\psi_n}{\psi_j}\,\frac1{1+\cE_n^j\phi_n}\right)_{n-j}.
$$
On the other hand, \eqref{psice} says that the right hand side of the
above equation is zero, which proves \eqref{0107}.
\end{proof}

\begin{proposition}\label{0072}
The system of functions $\frac{t^n}{\psi_n}$ is a Riesz basis for
$H^2$ if and only if the matrix $\cL^{-1}$ defines a bounded
operator on $\ell^2$, equivalently, $\cL$ is an isomorphism of
$H^2$.
\end{proposition}
\begin{proof}
Due to the natural isomorphism between $H^2$ and $\ell^2$,
$\frac{t^n}{\psi_n}$ is a Riesz basis for $H^2$ if and only if the
columns of $\cL^{-1}$ form a Riesz basis for $\ell^2$. In turn, the
columns of $\cL^{-1}$ form a Riesz basis for $\ell^2$ if and only if
both matrices $\cL^{-1}$ and $\cL$ define bounded operators on
$\ell^2$. By Proposition \ref{0070} $\cL$ is always a contraction.
\end{proof}
As a straightforward corollary of Theorem \ref{0137} and the second
part of Lemma \ref{0066}, we get
\begin{theorem}
$\phi$ is regular if and only if $$\left\{\begin{bmatrix}
{h_+} \\ {-\cH h_+}\end{bmatrix}, \quad h_+\in H^2_+\right\}$$ is
dense in $M_{\phi,+}$.
\end{theorem}
In view of Definition \ref{0068} we get the following
\begin{corollary}\label{0077}
The range of $\cL$ is dense in $H^2$ if and only if $\phi$ is
regular.
\end{corollary}
\begin{theorem}\label{0076}
$\cL^{-1}$ is a bounded operator on $H^2$ if and only if $\phi$ is
regular and $\Vert\cH\Vert<1$.
\end{theorem}
\begin{proof} By GLM equation \eqref{0063}, $\Vert\cH\Vert<1$ if and only if
\bg{\eq}\label{0078} \cL^*\cL\ge c I.
\end{\eq}
By Corollary \ref{0077}, $\phi$ is regular if and only if the range
of $\cL$ is dense in $H^2$. The latter along with \eqref{0078} is
equivalent to the boundedness of $\cL^{-1}$.
\end{proof}

\section{Helson-Szeg\H o class}\label{helszegoclass}
For a function $u$
$$
u=\sum\limits_{k=-\infty}^\infty c_k t^k
$$
harmonic conjugate $\tilde u$ is defined as
$$
\tilde u=-i\sum\limits_{k=-\infty}^\infty \textrm{\ sign\ }(k) c_k t^k,
$$
so the function $u+i\tilde u$ is ``analytic''. If $u$ is real, then
so is $\tilde u$. Note that $\tilde u$ does not depend on the
constant Fourier coefficient $u_0$. By the definition
$\widetilde{\widetilde{u}}=-u+u_0$.
\begin{definition}
We say that $w$ is a {\it positive Helson-Szeg\H o function} if it
admits a representation of the form
\begin{equation}\label{0022}
\begin{split}
    w =Ce^{u-\tilde v},\quad  &u,v \in L^\infty\ \text{\rm (real)}, \quad \sup v - \inf v <\pi,\\
    &u_0 =v_0=0,\quad C>0, \end{split}
\end{equation}
where $\tilde v$ is the harmonic conjugate of $v$, $u_0, v_0$ are
the constant Fourier coefficients. In this case we will say that the
absolutely continuous measure $\sigma(dt)=w(t) m(dt)\in {\bf HS}$.
\end{definition}
Unlike the standard convention $\|v\|<\pi/2$ we prefer to deal with
$$\sup v - \inf v <\pi $$
which is invariant under addition of any
constant. Conversely, if the latter holds, then
$$ \|v_c\|<\frac{\pi}2, \qquad v_c:=v+\frac{\sup v+\inf v}2\,. $$
\begin{definition}
A positive function $w$ is said to satisfy $A_2$ (or Hunt--Muckenhoupt--Wheeden)
condition if for all arcs $I\subset \bbT$ the following supremum is
finite
\begin{equation}\label{hmw}
    \sup_{I}\langle w\rangle_I\langle w^{-1}\rangle_I<\infty,\ \langle w\rangle_I:=\frac 1{|I|}
    \int\limits_{I}w(t) m(dt).
\end{equation}
Clearly $w\in A_2$ if and only if $1/w\in A_2$.
\end{definition}
The following classical theorem can be found , e.g., in
\cite[Lecture VIII]{Nik}.
\begin{theorem}[Helson--Szeg\H{o}]\label{hshmw}
The following conditions are equivalent
\begin{enumerate}
\item $w$ is a positive Helson--Szego function \eqref{0022};
\item $w$ satisfies the $A_2$ condition \eqref{hmw};
\item the angle is positive
between $H^2_{+,w}$ and $H^2_{-,w}$ in $L_w^2$:
$$
\left|\langle g_+, g_-\rangle_{w}\right|^2 \le\beta\ \|g_+\|_{w}^2
\cdot \|g_-\|_{w}^2 \,,\qquad \beta<1.
$$
\end{enumerate}
\end{theorem}
Here $H^2_{+,w}$ is the closure of analytic polynomials in $L^2_w$,
$H^2_{-,w}$ is the closure of conjugate-analytic polynomials that
vanish at the origin. It is known that for $w=|D|^2$
$$
H^2_{+,w}=D^{-1}H^2_+,\quad H^2_{-,w}=\ovl D^{\ -1}H^2_-.
$$
\begin{definition} We say that $s\in {\bf HS}$ if $s$ is a canonical symbol
of a Hankel operator $\cH$ with $\|\cH\|<1$.
\end{definition}
\begin{definition}\label{0136}
We say that a CMV matrix $\fA$ is of Helson--Szeg\H o class
($\fA\in{\bf HS}$) if $\cL^{-1}$ is a bounded operator, where $\cL$
is the transformation operator \eqref{0062}.
\end{definition}
In view of Theorem \ref{0076} $\fA\in{\bf HS}$ if and only of $\phi$
is regular and $\|\cH\|<~1$. Following Arov \cite{arov}, such
functions $\phi$ are called {\it strongly regular}. As a consequence
of the regularity of $\phi$, those $CMV$ matrices are automatically
absolutely continuous. Strongly regular functions form a proper
subclass of the regular ones.

\bigskip
The main result of this Section is
\begin{theorem}\label{th4.7}
There is a one-to-one correspondence between ${\bf HS}$ classes of
$CMV$ matrices $($Verblunsky coefficients$)$, spectral (probability) measures, and
scattering functions.
\end{theorem}
\begin{proof}
$\fA\in{\bf HS}\Longrightarrow s\in{\bf HS}$. By Definition \ref{0136},
$\fA\in{\bf HS}$ means that $\cL^{-1}$ is a bounded operator.
By Theorem \ref{0076}, the boundedness of $\cL^{-1}$
 is equivalent to the regularity of $\phi$ and
$\|\cH\|<1$. Since $\phi$ is regular, then, by Definition~\ref{0008}, $s$ is
canonical. Therefore, $s\in{\bf HS}$.

$\sigma\in{\bf HS}\Longrightarrow s\in{\bf HS}$.
Recall also that spectral measure in our context is always a
probability measure. Hence,
$$ w=\Re \frac{1-\bara\phi}{1+\bara\phi}=\frac{1-|\phi|^2}{|1+\bara\phi|^2} $$
with absolutely continuous $\frac{1-\bara\phi}{1+\bara\phi}$.
Assumption that $w$ ia a positive Helson--Szeg\H o function implies that $w$ is a Szeg\H o
function. Therefore,
$$ w=\frac{|\psi|^2}{|1+\bara\phi|^2}=|D|^2,\quad \text{where} \quad D=\frac{\psi}{1+\bara\phi} $$
In view of Theorem \ref{hshmw} (2), and since $D$ is outer, we get that
 $1/D\in H^2$. Since we also have that $\frac{1-\bara\phi}{1+\bara\phi}$
 is absolutely continuous, then,
 by Theorem~\ref{0018} (1),
 $\phi$ is regular. Therefore, $s$ is canonical.
 For $h_+\in H^2$ and $h_-\in H^2_-$ we have that
\bg{\eq}\label{0017}
\begin{split}
| \langle\cH h_+ , h_- \rangle | &= | \langle s h_+ , h_- \rangle |=
\left| \left\langle \frac{D}{\overline D} h_+ , h_-
\right\rangle\right| \\ &= \left| \left\langle \frac{1}{|D|^2} D h_+
, \overline D h_- \right\rangle \right|=|\langle D h_+ , \overline D
h_-\rangle|_{w^{-1}}.
\end{split}
\end{\eq}
Since $w^{-1}=|D|^{-2}\in A_2$, then, by Theorem \ref{hshmw},
\bg{\eq}\label{0079} |\langle D h_+ , \overline D
h_-\rangle|_{w^{-1}} \le \beta \Vert Dh_+ \Vert_{w^{-1}} \Vert \ovl
Dh_-\Vert_{w^{-1}}= \beta \Vert h_+ \Vert \Vert h_-\Vert ,\quad
\beta<1.
\end{\eq}
Therefore,  $\|\cH\|<1$. Hence, $s\in{\bf HS}$.

$s\in{\bf HS}\Longrightarrow \fA\in{\bf HS}$ and
$\sigma$ is a probability measure,
$\sigma\in{\bf HS}$.
Let $s$ be a canonical symbol of a Hankel operator $\cH$ with $\|\cH\|<1$:
$$
s=-\frac{\psi_\cH}{\overline\psi_\cH}\frac{\cE+\overline\phi_\cH}{1+\phi_\cH\cE},
$$
with $\cE$ unimodular constant and $\phi_H$ Arov-regular.
By Theorem \ref{thm1}, there exists a unique absolutely continuous $CMV$ matrix $\fA$ whose
scattering function is $s$, moreover this $\fA$ is regular.
$\a_{-1}$ and the (probability) spectral density $w$ are given by
$$
  \a_{-1}=\ovl\cE,\quad
  D=\frac{\psi_\cH}{1+\cE\phi_\cH};\quad w=|D|^2.
  $$
Verblunsky coefficients of $\fA$ are the Schur parameters of $\phi_\cH(\zeta)/\zeta$.
Since $\phi_\cH$ is regular and
$\|\cH\|<1$, then, by Theorem \ref{0076}, $\cL^{-1}$ is bounded, i.e., $\fA\in{\bf HS}$.

For $h_+\in H^2$ and $h_-\in H^2_-$ we have that
\bg{\eq}\label{0015}
| \langle\cH h_+ , h_- \rangle | \le \beta \Vert h_+ \Vert \Vert h_-\Vert ,
\quad\beta <1.
\end{\eq}
In view of \eqref{0017} and by Theorem \ref{hshmw}, \eqref{0015} implies \eqref{0079}.
Therefore, $|D|^2\in A_2$, meaning that $\sigma\in{\bf HS}$.
\end{proof}
\begin{remark}
The connection between strong regularity and $A_2$ condition was
observed and studied by D.~Arov and H.~Dym in \cite{AD1, AD2}. They
also extensively used that in their study on inverse spectral
problems for canonical systems of differential equations.
\end{remark}
\begin{remark}
Theorem \ref{th4.7} is contained in the preliminary version of the paper,
see \cite[Theorem 4.5, Proposition~4.7]{GKPY}. It was recently
observed in \cite[Theorem 6.3]{DFK}, that operator $\cL$ has a
multiplicative structure. This observation gives a hope that
the boundedness condition on $\cL^{-1}$ may be restated
as a {\it constructive} condition on the Verblunsky coefficients
via convergence of infinite products (series).
\end{remark}
\begin{definition}
We say that $s$ is a {\it unimodular Helson-Szeg\H o function}
if it admits a representation of the form
\begin{equation}\label{hsconds}
\begin{split}
    s=ce^{i(\tilde u+v)},\quad  &u,v\in L^\infty\ \text{\rm (real)}, \quad \sup v - \inf v
    <\pi,\\
    &u_0=v_0=0,\quad |c|=1, \end{split}
\end{equation}
where $\tilde u$ is the harmonic conjugate of $u$, $u_0, v_0$ are the constant Fourier coefficients.
\end{definition}
\begin{theorem}\label{0086}
Canonical symbols of Hankel operators with $\|\cH\|<1$ are
exactly unimodular Helson--Szeg\H o functions.
\end{theorem}
\begin{proof}
Let $s$ be a canonical symbol of the Hankel operator $\cH$ with
$\|\cH\|~<~1$, then, by Theorem \ref{th4.7}, the unique $w\in A_2$,
equivalently, $w$ is of the form \eqref{0022}. Then $w=|D|^2$, where
(taking into account our normalization $u_0=v_0=0$)
\bg{\eq}\label{0025} D=D(0){\Large e}^{^{\frac{u+i\tilde u
+i(v+i\tilde v)}{2}}},\ D(0)>0.
\end{\eq}
Therefore,
$$
s=-\bara\frac{D}{\overline D}=-\bara{\Large e}^{^{i(\tilde u+v)}},\
D(0)>0,
$$
and $s$ is a canonical symbol
of the Hankel operator $\cH$ with $\|\cH\|~<~1$.

Conversely, let
$s$ be a unimodular Helson--Szeg\H o function, i.e., it is of the form \eqref{hsconds}.
Then
$$
s=c{\Large e}^{^{i(\tilde u+v)}}=c\,\frac{D}{\overline D},
$$
where $\ |c|=1$, $D$ can be chosen as in \eqref{0025}. The
corresponding $w=|D|^2$ is of the form \eqref{0022}. Therefore,
$w\in A_2$ and, by Theorem \ref{th4.7}, $s$ is a canonical solution
of the Nehari problem with $\|\cH\|<1$. \end{proof}
\begin{remark}
In terms of representation \eqref{hsconds}, the unique solution of
the inverse scattering problem is given as
$$ \bara=-c, \quad w=Ce^{u-\tilde v}, \quad \int\limits_{\bbT} w(t)m(dt)=1. $$
\end{remark}

\section{B. Golinskii -- I. Ibragimov class}
\label{golibrag}
\begin{definition}
A function $g$ is in {\it Besov class} $B_2^{1/2}$ if
\begin{equation}\label{besov}
g=\sum\limits_{n=-\infty}^\infty g_n t^n, \qquad
\sum\limits_{n=-\infty}^\infty |n||g_n|^2<\infty . \end{equation}
Obviously,  $g\in B_2^{1/2}$ if and only if the harmonic conjugate
$\tilde g\in B_2^{1/2}$.
\end{definition}
Our arguments depend upon some classical results, mostly due to
V.~Peller \cite{pel} and S.~Khrushchev and V.~Peller \cite{KhP}; see
also \cite{peller}.
\begin{theorem}\cite[Proposition 6.1.11]{Sopuc}.\label{0024}
If $g\in B_2^{1/2}$, $g$ is real, then $e^{ig}\in B_2^{1/2}$ as
well.
\end{theorem}
\begin{theorem}\cite{pel}.\label{0142}
Conversely, every unimodular function s in Besov class is of the
form \bg{\eq}\label{0134} s=t^N e^{ig},
\end{\eq}
where $g$ is real, $g\in B_2^{1/2}$, $N$ is an integer called {\it
the index of $s$}. $N$ is determined uniquely, and $g$ is up to an
additive constant from $2\pi\bbZ$.
\end{theorem}
\begin{theorem}\cite{pel}.\label{0023}
Every function in $B_2^{1/2}$ has a representation of the form
$$
g=g_1+\tilde g_2,
$$
where $g_1$ and $g_2$ are continuous functions of Besov class. If $g$ is real, then
$g_1$ and $g_2$ are also real.
By means of trigonometric polynomial approximation, $C$-norm
of $g_1$ or $g_2$ can be made as little as we want.
\end{theorem}
\begin{theorem}\cite[Corollary 1.7, p. 72]{KhP}.\label{0133}
Let $s$ be a unimodular function. Let $T_s=P_+s|H^2:H^2\to H^2$. If
\bg{\eq}\label{0135} \ker T_s=\ker T_s^*=\{0\},
\end{\eq}
then the operators $\cH_{\overline s}^*\cH_{\overline s}$ and
$\cH_s^*\cH_s$ are unitarily equivalent. The equivalence is done
by the unitary factor $U$ in the polar decomposition of $T_s$
$$
T_s=U\sqrt{T_s^*T_s}.
$$
\end{theorem}
We start with the following
\begin{lemma}\label{0087}
Let $s$ be a unimodular function in Besov class of the index $N$.
Then $s$ is a unimodular Helson--Szeg\H o function if and only if
$N=0$.
\end{lemma}
\begin{proof}
By Theorem \ref{0142}, $s=t^Ne^{ig}$, $g$ is real, $g\in B_2^{1/2}$.
By Theorem \ref{0023} the function $\hat s=e^{ig}$ is a unimodular
Helson--Szeg\H o function (see \eqref{hsconds}), so by Theorem
\ref{0086} $\hat s$ is canonical.

If $N\not=0$, then, by Proposition \ref{noncanon}, $s=t^N\hat s$ is
not canonical, so, by Theorem \ref{0086}, $s$ is not a unimodular
Helson--Szeg\H o function. If $N=0$, then $s=\hat s$ is a unimodular
Helson--Szeg\H o function. The proof is complete.
\end{proof}
\begin{definition} We define Golinskii -- Ibragimov ({\bf GI}) classes
of $CMV$ matrices $($Verblunsky coefficients$)$, spectral
measures and scattering functions as follows
\begin{itemize}
\item[(1)]
{\bf GI} class of $CMV$ matrices
\bg{\eq}\label{0091}
\sum\limits_{n=0}^\infty n|a_n|^2<\infty,\quad\text{\rm equivalently}\quad
\prod\limits_{n=0}^\infty \rho_n^n<\infty .
\end{\eq}
We will also write $\fA\in {\bf GI}$.
\item[(2)] {\bf GI} class of spectral measures consists of absolutely continuous
 measures with density $w$ of the form
$w=e^g$, where $g$ is a real function in $B_2^{1/2}$. We will write
$\sigma\in {\bf GI}$. We will also say that the spectral data
$\{\sigma, \alpha_{-1}\}\in {\bf GI}$ if $\sigma\in {\bf GI}$.

\bigskip
\item[(3)] {\bf GI} class of scattering functions is the class
of functions $s$ of the form
 $s=e^{ig}$, where $g$ is a real function in $B_2^{1/2}$.
 We will also write $s\in {\bf GI}$.\bigskip
\end{itemize}
\end{definition}
\begin{lemma}\label{0141}
For {\bf GI} classes of $CMV$ matrices $($Verblunsky
coefficients$)$, spectral data and scattering functions the
following inclusions hold true $${\bf GI}\subset{\bf HS}.$$
\end{lemma}
\begin{proof}
Inclusion ${\bf GI}\subset{\bf HS}$ for spectral
measures and for scattering functions
follows from Theorem \ref{0023}. To prove the inclusion for CMV matrices
we show that \eqref{0091} implies boundedness of $\cL^{-1}$.

Let $\cL_m$ be the $m\times m$ principal block of the infinite
matrix $\cL$. Then the inverse matrix $(\cL_m)^{-1}$ will be the
$m\times m$ principal block of the infinite matrix $\cL^{-1}$
$$
(\cL_m)^{-1}=(\cL^{-1})_m\ .
$$
Due to this equality, we will use the notation $\cL_m^{-1}$.
Note that $\cL_m$ is a contraction. Indeed,
for $\ell_m$ a finite vector of length $m$,
$$
\Vert \cL_m\ell_m \Vert\le\Vert \cL\ell_m \Vert\le
\Vert \ell_m \Vert.
$$
Therefore, $\cL^{-1}_m$ is an expansion
\bg{\eq}\label{0044}
\cL_m^{-1 *}\cL_m^{-1}\ge I_m.
\end{\eq}
Now we get an upper bound on $\cL_m^{-1 *}\cL_m^{-1}$. Due to \eqref{0044}
$$
\cL_m^{-1 *}\cL_m^{-1}\le\det(\cL_m^{-1 *}\cL_m^{-1})I_m=|\det \cL_m^{-1}|^2 I_m
=\left(\prod\limits_{k=0}^m \frac{1}{\psi_k(0)}\right)^2I_m
$$
$$
=\left(\prod\limits_{k=0}^m \prod\limits_{j=k}^\infty\frac{1}{\rho_j}\right)^2 I_m
\le\left(\prod\limits_{k=0}^\infty \prod\limits_{j=k}^\infty\frac{1}{\rho_j}\right)^2 I_m
=\left(\prod\limits_{j=0}^\infty\frac{1}{\rho_j^{j+1}}\right)^2 I_m.
$$
Since the bound does not depend on $m$, we get that
the matrix $\cL^{-1}$
defines a bounded operator on $\ell^2$. The inclusion follows.
\end{proof}
\begin{theorem}\label{0021}
There is a one-to-one correspondence between {\bf GI}
classes of $CMV$ matrices $($Verblunsky coefficients$)$, spectral
data and scattering functions.
\end{theorem}
\begin{proof}
$\sigma\in {\bf GI}\Longleftrightarrow s\in {\bf GI}$
is straightforward. $s$ defines $\sigma$ and $\alpha_{-1}$
uniquely since $s$ is canonical (see Theorem \ref{thm1}).

$\fA\in {\bf GI}\Longrightarrow s\in {\bf GI}$.
 We consider $m\times m$ principal block of the GLM equation
\eqref{0063}
\begin{\eq}\label{0089}
I_m-(\cH^*\cH)_m=(\cL^{*}\cL)_m\ge\cL_m^{*}\cL_m.
\end{\eq}
We take the determinant of the both sides to get
\begin{\eq}\label{0090}
|\det\cL_m|^2\le\det(I_m-(\cH^*\cH)_m)\le e^{-\tr(\cH^*\cH)_m}
.
\end{\eq}
As we saw above
$$
|\det \cL_m|^2
=\left(\prod\limits_{k=0}^m {\psi_k(0)}\right)^2
$$
$$
=\left(\prod\limits_{k=0}^m \prod\limits_{j=k}^\infty{\rho_j}\right)^2
\ge\left(\prod\limits_{k=0}^\infty \prod\limits_{j=k}^\infty{\rho_j}\right)^2
=\left(\prod\limits_{j=0}^\infty{\rho_j^{j+1}}\right)^2>0.
$$
The latter bound is independent of $m$.
This and \eqref{0090} imply that
\bg{\eq}\label{0123}
\tr(\cH^*\cH)~<~\infty.
\end{\eq}
The trace is computed in
terms of $s$ as
$$
\tr(\cH^*\cH)=\sum\limits_{n=-\infty}^{-1}|n||c_n|^2,
$$
where $c_n$ are the Fourier coefficients of $s$. Therefore,
\bg{\eq}\label{0119}
P_-s\in B_2^{1/2}.
\end{\eq}

We show that actually
$$s\in B_2^{1/2}.$$
We are going to apply Theorem \ref{0133}. To this end we need to check \eqref{0135}.
The kernel of $T_s$ consists of the functions $h_+$ such that
$$
s h_+=h_-\in H^2_-.
$$
Since $s=-\overline\alpha_{-1}\frac{D}{\overline D}$, we get that
$$
-\overline\alpha_{-1}D h_+=\overline D h_-.
$$
The left-hand side is in $H^1$, the right--hand one is in $H^1_-$. Therefore, both sides equal $0$. Hence, the kernel of $T_s$ is trivial
\bg{\eq}\label{0122}
{\text{Ker\ }T_s}=\{0\}.
\end{\eq}
The kernel of $T_s^*$ consists of the functions $h_+$ such that
$$
\overline s h_+=h_-\in H^2_-.
$$
Since $s$ is canonical, by Proposition \ref{0065}, this equation has
only the trivial solution. Therefore, the kernel of $T_s^*$ is
trivial \bg{\eq}\label{0121} {\text{Ker\ }T_s^*}=\{0\}.
\end{\eq}

Due to \eqref{0122} and \eqref{0121} Therem \ref{0133} applies and we get that
$\cH_{\overline s}^*\cH_{\overline s}$ and
$\cH_s^*\cH_s$ are unitarily equivalent.
Therefore,
the eigenvalues of the operators $\cH_s^*\cH_s$ and $\cH_{\overline s}^*
\cH_{\overline s}$ coincide. The latter and \eqref{0123} imply that
\bg{\eq}\label{0124}
\tr(\cH_{\overline s}^*\cH_{\overline s})
=\tr(\cH_s^*\cH_s)
~<~\infty.
\end{\eq}
Hence, $P_-\overline s\in B_2^{1/2}$. We combine this with \eqref{0119} to get that
$s\in B_2^{1/2}$.

Since $s$ is a unimodular Helson--Szeg\H o function, then, by Lemma \ref{0087}, it has
the zero index.

$s\in {\bf GI}\Longrightarrow \fA\in {\bf GI}$. By Lemma \ref{0141}
$s\in {\bf GI}\Longrightarrow s\in {\bf HS}$. Then, by Theorem
\ref{th4.7}, there is a unique CMV matrix $\fA$ with this scattering
function and the corresponding operator $\cL^{-1}$ is bounded. The
latter allows us to rewrite GLM equation \eqref{0063} as
\begin{\eq}\label{0088}
(I-\cH^*\cH)^{-1}=(\cL^*\cL)^{-1}=\cL^{-1}\cL^{-1 *}.
\end{\eq}
Note that the first equality in \eqref{0088} makes sense once
$\|\cH\|<1$, while the second does for the ${\bf HS}$ class only!
We set
$$
 (I-\cH^*\cH)^{-1}=: I+\Delta,
 $$
 where $\Delta\ge 0$. $\tr(\cH^*\cH)~<~\infty$ if and only if $\tr\Delta<\infty$.
 Let $\Delta_m$ be $m\times m$ principal block of $\Delta$ (in the basis $t^n$). Then
 \begin{\eq}\label{0138}
 I_m+\Delta_m=(\cL^{-1}\cL^{-1 *})_m=\cL_m^{-1}\cL_m^{-1 *}.
 \end{\eq}
 The second equality here (compare with the inequality in \eqref{0089})
 holds true since now the left factor $\cL^{-1}$ is lower triangular and the right
 factor $\cL^{-1 *}$ is upper triangular. From \eqref{0138} we get
 $$
|\det\cL_m^{-1}|^2= \det(I+\Delta_m).
$$
Since
$$
1\le\det(I+\Delta_m) \le  e^{\tr\Delta_m}\le e^{\tr\Delta},
$$
\eqref{0091} follows.
 \end{proof}
\begin{remark} As we showed in the proof of
Theorem \ref{0021}, if $\cL^{-1}$ is bounded, then the following
version of Widom's formula holds true
$$
\det (I-\cH^*\cH)= \prod\limits_{j=0}^\infty{\rho_j^{2(j+1)}}\,.
$$
For the original Widom's formula see \cite{wid}, also \cite[Theorem
6.2.13]{Sopuc}.
\end{remark}
 \begin{remark} The equivalence
$\fA\in\GI\Longleftrightarrow\sigma\in\GI$ is the celebrated Strong
Szeg\H{o} Theorem (in Ibragimov's version). For the detailed
exposition see \cite[Chapter 6]{Sopuc}, where several independent
proofs are presented. Theorem \ref{0021} suggests another
alternate proof of this fundamental result via the scattering theory for CMV matrices.
\end{remark}
\begin{remark} In late 60s I. Ibragimov and V. Solev in their study
of classes of Gaussian stationary processes (see \cite[Chapter
4.4]{ibro}) came up with the class of spectral measures of the form
\begin{equation}\label{ibso}
\sigma(dt)=w(t)m(dt), \qquad w(t)=|P(t)|^2\,e^{h(t)},
\end{equation}
where $P$ is a polynomial of degree $N$ with all its zeros on the
unit circle, and $h$ is a real function from $B_2^{1/2}$. They
proved that scattering functions of measures\eqref{ibso} are exactly
unimodular functions $s$ from $B_2^{1/2}$ with ${\rm ind}s=N$. Note
that in this class solution of the inverse scattering problem is not
unique. A description of the corresponding CMV matrices (similar to
\eqref{0091}) is not known.
\end{remark}
\begin{example}\label{jac} This example shows that the inclusion
${\bf GI}\subset{\bf HS}$ is proper.
We consider the Jacobi weight for the unit circle
$$ w(t)= C|t-1|^{2\gamma_1}\,|t+1|^{2\gamma_2}, \quad
D(z)=C^{1/2}(1-z)^{\gamma_1}(1+z)^{\gamma_2}, \quad
\gamma_{1,2}>-\frac12
$$ that enters the theory several times. First,
for the choices of the parameters $\gamma_1=0, \gamma_2=2$ and
$\gamma_1=2, \gamma_2=0$
we get two different
weights $w_{\pm}=C|t\pm 1|^4$ with the Szeg\H{o} functions
$D_{\pm}(z)=C^{1/2}(1\pm z)^2$, that have the same
scattering function $s=t^2$. Next, $w\in A_2$ if and only if
$|\gamma_k|<1/2$. Finally, the Verblunsky coefficients are known
explicitly
$$ a_n=-\frac{\gamma_1-(-1)^n\,\gamma_2}{n+1+\gamma_1+\gamma_2}\,,
\quad n=0,1,\ldots$$ so $w$ is never in $\GI$ unless
$\gamma_1=\gamma_2=0$.
\end{example}

\bibliographystyle{amsplain}

\bigskip

{\em Mathematics Division, Institute for Low Temperature Physics and

Engineering, 47 Lenin ave., Kharkov 61103, Ukraine}

{E-mail: leonid.golinskii@gmail.com}

\bigskip
{\em Department of Mathematics, University of Massachusetts Lowell,

Lowell, MA, 01854, USA}

E-mail: Alexander\_Kheifets@uml.edu

\bigskip
{\em Institute for Analysis, Johannes Kepler University of Linz,

A-4040 Linz, Austria}

E-mail: Petro.Yudytskiy@jku.at

\end{document}